\documentclass[12pt, reqno]{amsart}
\usepackage{amsmath, amsthm, amscd, amsfonts, amssymb, graphicx, color}
\usepackage[bookmarksnumbered, colorlinks, plainpages]{hyperref}
\hypersetup{colorlinks=true,linkcolor=red, anchorcolor=green, citecolor=cyan, urlcolor=red, filecolor=magenta, pdftoolbar=true}

\newtheorem{theorem}{Theorem}[section]
\newtheorem{lemma}[theorem]{Lemma}
\newtheorem{proposition}[theorem]{Proposition}
\newtheorem{corollary}[theorem]{Corollary}

\theoremstyle{definition}

\theoremstyle{remark}

\newcommand{\tr}{{\rm{tr}}}

\numberwithin{equation}{section}

\newtheorem*{theorem*}{Theorem}

\begin{document}
\title[Quadratic interpolation of the Heinz means]{ Quadratic interpolation of the Heinz means}

\author[F. Kittaneh]{Fuad Kittaneh}
\address{Department of Mathematics, The University of Jordan, Amman, Jordan}
\email{fkitt@ju.edu.jo}
\author[M.S. Moslehian]{Mohammad Sal Moslehian}
\address{Department of Pure Mathematics, Center Of Excellence in Analysis on Algebraic Structures (CEAAS), Ferdowsi University of Mashhad, P. O. Box 1159, Mashhad 91775, Iran}
\email{moslehian@um.ac.ir}
\author[M. Sababheh]{Mohammad Sababheh}
\address{Department of Basic Sciences, Princess Sumaya University for Technology, Amman, Jordan}
\email{sababheh@psut.edu.jo, sababheh@yahoo.com}


\maketitle
\begin{abstract}
The main goal of this article is to present several quadratic refinements and reverses of the well known Heinz inequality, for numbers and matrices, where the refining term is a quadratic function in the mean parameters. The proposed idea introduces a new approach to these inequalities, where polynomial interpolation of the Heinz function plays a major role. As a consequence, we obtain a new proof of the celebrated Heron-Heinz inequality proved by Bhatia, then we study an optimization problem to find the best possible refinement. As applications, we present matrix versions including unitarily invariant norms, trace and determinant versions.
\end{abstract}

\section{introduction}

The celebrated Heinz inequality states that
\begin{equation}\label{original_heinz_intro}
2\sqrt{ab}\leq a^{t}b^{1-t}+a^{1-t}b^{t}\leq a+b, a,b>0,\;0\leq t\leq 1.
\end{equation}
 In \cite{B1}, a matrix version of this  inequality was shows as follows
\begin{equation}\label{matrix_heinz_intro}
2\vert\vert\vert A^{\frac{1}{2}}XB^{\frac{1}{2}}\vert\vert\vert \leq\vert\vert\vert A^{t}XB^{1-t}+A^{1-t}XB^{t}\vert\vert\vert \leq \vert\vert\vert AX+XB\vert\vert\vert , 0\leq t\leq 1,
\end{equation}
where $X\in\mathbb{M}_n$, the algebra of all $n\times n$ complex matrices, and $A,B\in\mathbb{M}_n^{+}$, the cone of positive semidefinite matrices in $\mathbb{M}_n.$ In the setting of matrices, the notation $\vert\vert\vert \cdot\vert\vert\vert $ will be used for an arbitrary unitarily invariant norm on $\mathbb{M}_n.$ Recall that these are norms on $\mathbb{M}_n$ with the property $\vert\vert\vert UXV\vert\vert\vert =\vert\vert\vert X\vert\vert\vert $ for all $X\in\mathbb{M}_n$ and any unitary matrices $U,V\in\mathbb{M}_n.$

In the past few years, a considerable attention has been put towards refining or reversing these inequalities, and some  related inequalities. For example, in \cite{kitt}, the convexity of the function $t\mapsto \vert\vert\vert \leq\vert\vert\vert A^{t}XB^{1-t}+A^{1-t}XB^{t}\vert\vert\vert $ was utilized to find some refining terms of \eqref{matrix_heinz_intro}. Then in \cite{conde,sabjmi}, further refinements were obtained, modeling the same idea of \cite{kitt}; see also \cite{BM, KMSC}. For example, it was shown in \cite{kitt} that $$f(t)\leq f(t/2)\leq \frac{1}{t}\int_{0}^{t}f(x)dx\leq \frac{f(0)+f(1)}{2}\leq f(0),\;0\leq t\leq \frac{1}{2}$$ where $f(t)=\vert\vert\vert A^{t}XB^{1-t}+A^{1-t}XB^{t}\vert\vert\vert .$
On the other hand \cite{kitt_man,kittanehmanasreh} presented the refinement and reverse
\begin{align*}
\|A^{t}XB^{1-t}&+A^{1-t}XB^{t}\|_2^2+r(t)\|AX-XB\|_2^2\\
&\leq \|AX+XB\|_2^2\\
&\leq \|A^{t}XB^{1-t}+A^{1-t}XB^{t}\|_2^2+R(t)\|AX-XB\|_2^2, 0\leq t\leq 1,
\end{align*}
where $r(t)=\min\{t,1-t\}, R(t)=\max\{t,1-t\}$ and $\|\cdot\|_2$ is the Hilbert-Schmidt norm defined, for $A\in \mathbb{M}_n$, as follows
$$\|A\|_2^2=\sum_{i,j}|a_{ij}|^2=\tr(AA^*).$$

In the above refinements, and many others, $r(t)$ or $R(t)$ term is  linear in the parameter $t.$

Earlier, Bhatia \cite{bhatiaint} showed that, for $\alpha(t)=(1-2t)^2,$
\begin{equation}\label{heinz_heron_bhatia_intro}
H_t(a,b)\leq K_{\alpha(t)}(a,b), 0\leq t\leq 1,
\end{equation}
 where
$$H_t(a,b)=\frac{a^{t}b^{1-t}+a^{1-t}b^{t}}{2}\;{\text{and}}\;K_{t}(a,b)=(1-t)\sqrt{ab}+t\frac{a+b}{2}$$ are the Heinz and Heron means, respectively.

Notice that both the Heinz means $H_t(a,b)$ and the Heron means $K_t(a,b)$ interpolate between the geometric mean $a\#b:=\sqrt{ab}$ and the arithmetic mean $a\nabla b:=\frac{a+b}{2}.$\\
Inequality \eqref{heinz_heron_bhatia_intro} attracted researchers who investigated this inequality and some possible matrix versions. We refer the reader to  \cite{singh1, zoul} for some nice discussion and history of these inequalities.\\
Rewriting \eqref{heinz_heron_bhatia_intro}, we obtain the following refinement of the Heinz inequality
$$H_t(a,b)+4t(1-t)(a\nabla b-a\#b)\leq a\nabla b.$$ This last refinement has been explored recently in \cite{krnic}, where some matrix versions were obtained.

Our motivation of the current work begins with this last inequality and its relation to \eqref{heinz_heron_bhatia_intro}. So, our first concern is why $\alpha(t)$ is given this way, and is there any alternative? It turns out \eqref{heinz_heron_bhatia_intro} follows from a more general inequality that treats quadratic interpolation of the Heinz means. More precisely, if we let $H(t):=H_t(a,b)$ and we find the quadratic polynomial interpolating $H$ at $t=0,\frac{1}{2},1$, we obtain $K_{\alpha(t)}.$ Therefore, \eqref{heinz_heron_bhatia_intro} has its geometric meaning now.\\
But then, if this is the origin of \eqref{heinz_heron_bhatia_intro}, what about taking the quadratic polynomial interpolating $H_t(a,b)$ at $0,\tau,1$, for an arbitrary value $\tau\in (0,1).$ This idea will be the main work in this paper, where we describe these polynomials and their relation to the Heinz means. Then, we discuss the ``best'' possible choice of $\tau$. This decision will depend on the error between the Heinz means $H(t)$ and the quadratic polynomial $F_{\tau}(t)$. We will show that the 1-norm difference between $H(t)$ and $F_{\tau}(t)$ is minimized at the unique root of $8\tau^3-12\tau^2+1=0$ in $\left(0,\frac{1}{2}\right).$ This is an interesting result because this $\tau$ is independent of $a$ and $b.$ Our numerical experiments show that  other norms are minimized at values that depend on $a$ and $b$, which makes the 1-norm an interesting case.

To prove our results, we need to prove monotonicity of certain functions.

Our first main result in this paper states that if $\tau\in(0,1), r(\tau)=\min\{\tau,1-\tau\}, R(\tau)=1-r(\tau)$ and $ \nu\leq r(\tau)$ or
$\nu\geq R(\tau),$ we have the inequality
$$\frac{a\nabla b-H_{\tau}(a,b)}{\tau(1-\tau)}\leq \frac{a\nabla b-H_{\nu}(a,b)}{\nu(1-\nu)},$$ while we have the reversed inequality when  $r(\tau)\leq\nu\leq R(\tau)$. Then letting $\tau=\frac{1}{2}$ implies \eqref{heinz_heron_bhatia_intro}. Therefore, this is a generalization and a new proof of \eqref{heinz_heron_bhatia_intro}. This last inequality can be thought of as a quadratic refinement and  reverse of the Heinz inequality. Then this idea is explored to obtain squared versions and multiplicative versions.

Once these numerical results are proved, we present their matrix versions, where unitarily invariant norms, trace and determinants are involved.
Some matrix versions are as follows. For certain $\tau,\nu$, one has
\begin{align*}
\|A^{\nu}XB^{1-\nu}+A^{1-\nu}XB^{\nu}\|_2^2&+\frac{\nu(1-\nu)}{\tau(1-\tau)}\left[\|AX+XB\|_2^2-\|A^{\tau}XB^{1-\tau}+A^{1-\tau}XB^{\tau}\|_2^2\right]\\
&\leq \|AX+XB\|_2^2,
\end{align*}
which is a quadratic refinement of the matrix Heinz inequality. If we let $\tau=\frac{1}{2}$ in this inequality, we obtain a recent result of Krni\'c \cite{krnic}.
Another matrix version for any unitarily invariant norm will be
\begin{align*}
&\vert\vert\vert A^{\nu}XB^{1-\nu}+A^{1-\nu}XB^{\nu}\vert\vert\vert \\
&+\frac{\nu(1-\nu)}{\tau(1-\tau)}\left[\left(\vert\vert\vert AX\vert\vert\vert +\vert\vert\vert XB\vert\vert\vert \right)-
\left(\vert\vert\vert AX\vert\vert\vert ^{\tau}\vert\vert\vert XB\vert\vert\vert ^{1-\tau}+\vert\vert\vert AX\vert\vert\vert ^{1-\tau}\vert\vert\vert XB\vert\vert\vert ^{\tau}\right)\right]\\
&\leq \vert\vert\vert AX\vert\vert\vert +\vert\vert\vert XB\vert\vert\vert .
\end{align*}
Further results about the determinant and the trace will be presented too.

For the notations adopted in this paper, we use
$$a\nabla_tb=(1-t)a+tb\;{\text{and}}\;a\#_{t}b=a^{1-t}b^{t}$$ for the weighted scalar arithmetic and geometric means, while
$$A\nabla_{t}B=(1-t)A+tB\;{\text{and}}\;A\#_{t}B=A^{\frac{1}{2}}\left(A^{-\frac{1}{2}}BA^{-\frac{1}{2}}\right)^{t}A^{\frac{1}{2}}$$ will be used for the matrix arithmetic and geometric means, when $A,B\in\mathbb{M}_n^{++}$, the cone of positive definite matrices in $\mathbb{M}_n.$

\section{Main Results}

\subsection{Scalar Results} In this part of the paper, we present the scalar results that we need to accomplish the matrix versions. The main tool in proving the scalar results is some delicate and tricky computations.

\subsubsection{The main scalar results}

\begin{theorem}\label{first_monoton_theorem}
For $c>0,$ let $$f(t)=\frac{1+c-(c^t+c^{1-t})}{t(1-t)}.$$ Then $f$ is decreasing on $\left(0,\frac{1}{2}\right)$ and is increasing on
$\left(\frac{1}{2},1\right).$
\end{theorem}
\begin{proof}
Notice that $f(t)=g(t)+g(1-t),$ where $$g(t)=\frac{1\nabla_{t}c-1\#_{t}c}{t(1-t)}.$$ We prove that $g$ is convex on $(0,1),$ then we use
this observation to prove the stated facts about $f$. Direct Calculus computations show that
$$g''(t)=-\frac{h(c)}{t^3(1-t)^3},$$ where
$$h(c)=-2 + 2 t (3 + t (-3 + t - c t)) +  c^t (2 + 6 (-1 + t) t + (-1 + t) t \log c (2 - 4 t + (-1 + t) t \log c)).$$ Then
$$h'(c)=\frac{t^3}{c}k(t)\;{\text{where}}\;k(t)=-2 c + c^t (2 + (-1 + t) \log c (-2 + (-1 + t) \log c))$$
and $$k'(t)=c^t(1-t)^2\log^3c.$$
We discuss two cases:\\
{\bf{Case I:}} If $c>1,$ then clearly $k'(t)>0$ and $k$ is an increasing function of $t\in (0,1).$
In particular, $k(t)\leq k(1)=0$,  hence $h'(c)<0$ and $h$ is decreasing in $c\in (1,\infty).$ That is, $h(c)\leq h(1)=0.$ Now since $h(c)\leq 0$
and $0<t<1,$ we infer that $g''(t)\geq 0,$ when $c>1.$\\
{\bf{Case II:}} If $0<c<1,$ then clearly $k'(t)<0$ and $k$ is decreasing in $t\in (0,1).$ That is, $k(t)\geq k(1)=0$ and $h'(c)\geq 0$.
 Since $h$ is increasing in $c\in (0,1),$ we have $h(c)\leq h(1)=0$ and hence, $g''(t)\geq 0.$

Thus, we have shown that, for any $c>0$, $g$ is convex on $(0,1).$ Now since $f(t)=g(t)+g(1-t),$ $f$ is clearly convex. Notice that $f(0)=f(1)$.
Since $f$ is convex it follows that either $f$ is monotone on $(0,1)$ or is decreasing on $(0,t_0)$ and is increasing on $(t_0,1)$ for some $0<t_0<1.$
But since $f(0)=f(1)$, we have the later case. Thus, there exists $t_0\in (0,1)$ with the above monotonicity property. We assert that $t_0=\frac{1}{2}.$
 Notice first that $f'\left(\frac{1}{2}\right)=0.$ By Taylor theorem, for any $t\in (0,1)$,
 there exists $\xi_t$ between $\frac{1}{2}$ and $t$ such that
 $$f(t)=f\left(\frac{1}{2}\right)+\frac{f''(\xi_t)}{2}\left(t-\frac{1}{2}\right)^2\geq f\left(\frac{1}{2}\right),$$ where the last inequality follows
 from the fact that $f''>0.$ This proves that $f$ attains its minimum at $t_0=\frac{1}{2}.$ This completes the proof.
\end{proof}

  \begin{corollary}
  Let $a,b>0$ and $0\leq t\leq 1.$ Then the following quadratic refinement of Heinz inequality holds
\begin{equation}\label{first_heinz}
\left(a^tb^{1-t}+a^{1-t}b^t\right)+4t(1-t)(\sqrt{a}-\sqrt{b})^2\leq a+b.
\end{equation}
  \end{corollary}
\begin{proof}
 For $c=\frac{a}{b},$ let $f(t)=\frac{1+c-(c^t+c^{1-t})}{t(1-t)}.$ Then $f$ attains its minimum at $t_0=\frac{1}{2}.$ That is,
$f\left(t\right)\geq f\left(\frac{1}{2}\right).$ Simplifying this simple inequality implies the result.
\end{proof}
In particular, we obtain the following Heinz-Heron mean inequality. The proof follows immediately by simplifying \eqref{first_heinz}.
\begin{corollary}
 Let $a,b>0, 0\leq t\leq 1$ and let $H_t(a,b)$ and $K_t(a,b)$ denote the Heinz and Heron means respectively. Then
\begin{equation}\label{heron_heinz_bhatia}
H_t(a,b)\leq K_{\alpha(t)}(a,b),\;{\text{where}}\;\alpha(t)=1-4t(1-t).
\end{equation}
  \end{corollary}
Thus, this is another proof of the well known inequality \eqref{heinz_heron_bhatia_intro} proved by Bhatia in \cite{bhatiaint}. In fact, even this follows from a more
general comparison of the Heinz means. The following result presents a quadratic refinement and reverse of the Heinz inequality.
\begin{corollary}\label{ref_rever_heinz_cor}
 Let $a,b>0$ and $0<\nu<\tau\leq \frac{1}{2}.$ Then
\begin{equation}\label{second_heinz}
\frac{a\nabla b-H_{\tau}(a,b)}{\tau(1-\tau)}\leq \frac{a\nabla b-H_{\nu}(a,b)}{\nu(1-\nu)}.
\end{equation}
The inequality is reversed if $\frac{1}{2}\leq \nu<\tau<1.$
\end{corollary}
\begin{proof}
For $c=\frac{a}{b},$ the function $f(t)=\frac{1+c-(c^{t}+c^{1-t})}{t(1-t)}$ is decreasing when $t<\frac{1}{2}$ and is increasing when $t>\frac{1}{2}.$
Now if $\nu<\tau\leq \frac{1}{2}$, we have $f(\tau)\leq f(\nu)$, which implies the desired inequality in this case. On the other hand, if
$\frac{1}{2}\leq \nu<\tau,$ then $f(\tau)\geq f(\nu),$ which implies the reversed inequality.
\end{proof}

Since the functions we are dealing with are symmetric about $t=\frac{1}{2},$ a full comparison is as follows.
\begin{corollary}\label{full_comp_heinz}
 Let $a,b>0$ and fix $\tau\in (0,1).$ Then
$$\frac{a\nabla b-H_{\tau}(a,b)}{\tau(1-\tau)}\leq \frac{a\nabla b-H_{\nu}(a,b)}{\nu(1-\nu)}$$ for $ \nu\leq r(\tau)$ and
$\nu\geq R(\tau).$ On the other hand if $r(\tau)\leq\nu\leq R(\tau)$, the inequality is reversed.

\end{corollary}
In fact the above inequalities have their own geometric meaning, as follows. Let $H(t)=H_t(a,b)=\frac{a^tb^{1-t}+a^{1-t}b^t}{2}$
 and fix any $\tau\in (0,1).$ Using the points $(0,H(0)), \left(\tau,H(\tau)\right)$ and $(1,H(1))$, we may find a quadratic polynomial that
 interpolates $f$ at $0,\tau,1.$ Consider the  function
\begin{equation}\label{definition_F}
F_{\tau}(t)=a\nabla b-\frac{a\nabla b-H_{\tau}(a,b)}{\tau(1-\tau)}t(1-t).
\end{equation}
 Notice that, when $\tau$ is fixed, $F_{\tau}$ is a quadratic polynomial which
coincides with $H$ at $t=0,\tau,1$. That is, $F_{\tau}$ is the quadratic interpolating polynomial of $H_t(a,b).$

From Corollary \ref{full_comp_heinz}, it follows that $H_t(a,b)\geq F_{\tau}(t)$ when $r(\tau)\leq t\leq R(t)$ and $H_t(a,b)\leq F_{\tau}(t)$ when $t\leq r(\tau)$
or $t\geq R(\tau).$ Notice that when $\tau=\frac{1}{2}$, we are left with $H_t(a,b)\leq F_{\frac{1}{2}}(t), 0\leq t\leq 1,$ which is the known comparison between
the Heinz and Heron means!

Our next target is to present a squared version of these refinements. This will help prove some Hilbert-Schmidt norm inequalities for matrices.
\begin{proposition}\label{squared_heinz_prop}
Let $a,b>0$ and $0<\nu,\tau<1.$ If $\nu\leq r(\tau)$ or $\nu\geq R(\tau)$, then
\begin{equation}\label{square_heinz_eq}
\frac{(a\nabla b)^2-H_{\tau}(a,b)^2}{\tau(1-\tau)}\leq \frac{(a\nabla b)^2-H_{\nu}(a,b)^2}{\nu(1-\nu)}.
\end{equation}
On the other hand, if $r(\tau)\leq\nu\leq R(\tau)$, the inequality is reversed.
\end{proposition}
\begin{proof}
Notice that, when $\nu\leq r(\tau)$ or $\nu\geq R(\tau)$,
\begin{align*}
4\left[(a\nabla b)^2-H_{\nu}(a,b)^2\right]&=(a+b)^2-(a^{\nu}b^{1-\nu}+a^{1-\nu}b^{\nu})^2\\
&= 2\left[a^2\nabla b^2-H_{\nu}(a^2,b^2)\right]\\
&\geq 2\frac{\nu(1-\nu)}{\tau(1-\tau)}\left[a^2\nabla b^2-H_{\tau}(a^2,b^2)\right]\\
&=\frac{\nu(1-\nu)}{\tau(1-\tau)}\left[(a+b)^2-(a^{\tau}b^{1-\tau}+a^{1-\tau}b^{\tau})^2\right].
\end{align*}
Then dividing by 4 implies the desired result when $\nu\leq r(\tau)$ or $\nu\geq R(\tau)$. The other case follows similarly.
\end{proof}

Notice that the above refinements and reverses of Heinz inequality have been found using the monotonicity of the function $f(t)=\frac{1+c-(c^{t}
+c^{1-t})}{t(1-t)}.$ Convexity of this function, which we have shown in Theorem \ref{first_monoton_theorem} implies the following reverse.
\begin{corollary}
Let $a,b>0$ and $0\leq t\leq 1.$ Then
\begin{equation}\label{log_reverse_heinz}
\left(a^{t}b^{1-t}+a^{1-t}b^{t}\right)+t(1-t)(b-a)\log\frac{b}{a}\geq a+b.
\end{equation}
\end{corollary}
\begin{proof}
For $c=\frac{a}{b},$ the function $f(t)=\frac{1+c-(c^{t}+c^{1-t})}{t(1-t)}$ is convex. Therefore, $f(t)\leq (1-t)f(0)+t f(1).$ Simplifying
this  inequality implies the result.
\end{proof}

In \cite{kitt_man}, a reversed version of Heinz inequality was proved as follows
\begin{equation}\label{kitt_man_reverse}
\left(a^{t}b^{1-t}+a^{1-t}b^{t}\right)^2+2\max\{t,1-t\}(a-b)^2\geq (a+b)^2.
\end{equation}

Numerical experiments show that neither \eqref{log_reverse_heinz} nor \eqref{kitt_man_reverse} is uniformly better than the other. However, these
experiments show that, for most values of $t$, \eqref{log_reverse_heinz} is better than \eqref{kitt_man_reverse} when $\frac{a}{b}$ is relatively small
and \eqref{kitt_man_reverse} is better when $\frac{a}{b}$ is large. In fact, a squared logarithmic-refinement maybe obtained as follows.
\begin{proposition}\label{squared_logarithmic}
 Let $a,b>0$ and $0\leq t\leq 1.$ Then
$$(a^{t}b^{1-t}+a^{1-t}b^{t})^2+2t(1-t)(b^2-a^2)\log\frac{b}{a}\leq (a+b)^2.$$
\end{proposition}
\begin{proof}
 Notice that, utilizing \eqref{log_reverse_heinz},
\begin{align*}
 (a+b)^2-(a^{t}b^{1-t}+a^{1-t}b^{t})^2&=(a^2+b^2)-\left((a^2)^t(b^2)^{1-t}+(a^2)^{1-t}(b^2)^t\right)\\
&\leq t(1-t)(b^2-a^2)\log\frac{b^2}{a^2}.
\end{align*}
This completes  the proof.
\end{proof}

The above refinement are all additive versions, where the refining term is added to one side of the inequality. Multiplicative versions can be found as follows.
\begin{lemma}
 For $c>0$, let $$f(t)=\left(\frac{1+c}{c^t+c^{1-t}}\right)^{\frac{1}{t(1-t)}}.$$ Then $f$ is increasing on $\left(0,\frac{1}{2}\right)$ and is decreasing on $\left(\frac{1}{2},1\right).$
\end{lemma}
\begin{proof}
 We prove that $f$ is increasing on $\left(0,\frac{1}{2}\right)$, then the conclusion for the other interval follows by symmetry of $f$. Thus, for $0<t<\frac{1}{2},$ let $F(t)=\log f(t).$ That is,
$$F(t)=\frac{\log(c+1)-\log(c^t+c^{1-t})}{t(1-t)}.$$ Then
$$F'(t)=\frac{G(t)}{(c+c^{2t})t^2(1-t)^2},$$ where
\begin{eqnarray*}
G(t)&=&(c-c^{2t})t(1-t)\log c+(c+c^{2t})(2t-1)\left(\log(c+1)-\log(c^{t}+c^{1-t})\right)\\
&=&(c-c^{2t})g(t),
\end{eqnarray*}
for $$g(t)=t(1-t)\log c+\frac{(c+c^{2t})(2t-1)\left(\log(c+1)-\log(c^{t}+c^{1-t})\right)}{c-c^{2t}}.$$ Further, we have
$$g'(t)=2\frac{\log(c+1)-\log(c^t+c^{1-t})}{(c-c^{2t})^2}h(t),$$ where
$$h(t)=c^2-c^{4t}+2c^{1+2t}(2t-1)\log c.$$ Finally we have
$$h'(t)=4c^{2t}k(t)\log c \;{\text{where}}\;k(t)=c-c^{2t}+(2c\;t-c)\log c$$ and $$k'(t)=2(c-c^{2t})\log c.$$

Now we treat two cases, based on whether $c>1$ or $c<1.$\\

If $c>1$, then clearly $k'(t)>0$ because $0<t<\frac{1}{2}$, and $k$ is increasing. Hence, $k(t)\leq k\left(\frac{1}{2}\right)=0$ and $h$ is decreasing. Therefore, $h(t)\geq h\left(\frac{1}{2}\right)=0$ and $g$ is increasing. Since $g$ is increasing, we have $g(t)\geq g(0)=0$, and hence $G\geq 0.$ This shows that $F'(t)\geq 0$ when $c>1$ and $0<t\leq\frac{1}{2}.$

If $0<c<1,$ then $k'(t)>0$ and $k\leq 0$. Hence $h'>0$ and $h(t)\leq h\left(\frac{1}{2}\right)=0$. That is $g'<0$ and $g(t)\leq g(0)=0$. Since $g(t)\leq 0, 0<t<\frac{1}{2}$ and $0<c<1$, it follows that $G(t)\geq 0$ and $F'(t)\geq 0.$

Thus, we have shown that for $0<t<\frac{1}{2}$ and $c>0$, we have $F'(t)\geq 0.$ This completes the proof.
\end{proof}

In particular, $f(t)=\left(\frac{1+c}{c^t+c^{1-t}}\right)^{\frac{1}{t(1-t)}}$ attains its maximum at $t_0=\frac{1}{2}.$ This entails the following reversed version of the Heinz inequality.
\begin{corollary}
Let $a,b>0$ and let $0\leq t\leq 1.$ Then
$$a+b\leq \left(\frac{a\nabla b}{a\#b}\right)^{4t(1-t)}\left(a^{t}b^{1-t}+a^{1-t}b^{t}\right).$$
\end{corollary}
A full Comparison can be given as follows.
\begin{corollary}\label{full_multi_arith_geo_cor}
Let $a,b>0$ and let $0<\nu,\tau<1.$ If $\nu\leq r(\tau)$ or $\nu\geq R(\tau)$, then
\begin{equation}\label{full_multi_arith_geo}
\left(\frac{a\nabla b}{H_{\nu}(a,b)}\right)^{\frac{1}{\nu(1-\nu)}}\leq \left(\frac{a\nabla b}{H_{\tau}(a,b)}\right)^{\frac{1}{\tau(1-\tau)}}.
\end{equation}
 On the other hand, if $r(\tau)\leq \nu\leq R(\tau),$ the inequality is reversed.
\end{corollary}
Notice that \eqref{full_multi_arith_geo} maybe though of as a refinement and a reverse of the Heinz inequality $H_{\nu}(a,b)\leq a\nabla b,$ if written as
$$H_{\tau}(a,b)\left(\frac{a\nabla b}{H_{\nu}(a,b)}\right)^{\frac{\tau(1-\tau)}{\nu(1-\nu)}}\leq a\nabla b\leq \left(\frac{a\nabla b}{H_{\tau}(a,b)}\right)^{\frac{\nu(1-\nu)}{\tau(1-\tau)}}H_{\nu}(a,b).$$

In fact, Corollary \ref{full_multi_arith_geo} does not provide a refinement and a reverse of the Heinz inequality $H_{\nu}(a,b)\leq a\nabla b,$ but it also provides a refinement of the first inequality of \eqref{original_heinz_intro}, as follows. Letting $\tau=\frac{1}{2}$ in \eqref{full_multi_arith_geo}, we have
\begin{align*}
2\sqrt{ab}\leq (a+b)\left(\frac{a^{\nu}b^{1-\nu}+a^{1-\nu}b^{\nu}}{a+b}\right)^{\frac{1}{4\nu(1-\nu)}}, 0<\nu<1.
\end{align*}
Now noting that $\frac{a^{\nu}b^{1-\nu}+a^{1-\nu}b^{\nu}}{a+b}\leq 1$ and $\frac{1}{4\nu(1-\nu)}\geq 1,$ we have
\begin{align*}
2\sqrt{ab}&\leq(a+b)\left(\frac{a^{\nu}b^{1-\nu}+a^{1-\nu}b^{\nu}}{a+b}\right)^{\frac{1}{4\nu(1-\nu)}}\\
&\leq(a+b)\frac{a^{\nu}b^{1-\nu}+a^{1-\nu}b^{\nu}}{a+b}=a^{\nu}b^{1-\nu}+a^{1-\nu}b^{\nu},\;0<\nu<1.
\end{align*}

\subsubsection{The best quadratic interpolator of the Heinz means}

We have observed in the previous subsection that the Heinz inequality can be refined or reversed by looking at the quadratic polynomial interpolating $H_t$ at $t=0,\tau,1$ for any choice of $0<\tau<1.$ Moreover, we have seen that the celebrated result of Bhatia \cite{bhatiaint} about the comparison between the Heinz and Heron means happens to be a special case of this general interpolation idea, taking $\tau=\frac{1}{2}.$

In this part of the paper, we try to describe the ``best" quadratic polynomial $F_{\tau}$ that interpolates $H_t$. Thus, we are searching for $\tau$ that minimizes the error $\|H_t-F_{\tau}\|$, for some norm. We present this best interpolator using the norm $\|\cdot\|_1.$ In particular, we show that $\|H_t-F_{\tau}\|_1$ will have its minimum value when $\tau=\tau^*$, where $\tau^*$ is the unique root of $8\tau^3-12\tau^2+1=0$ between $0$ and $\frac{1}{2}.$ Thus, $\tau^*\approx 0.326352.$ Simple calculations show that this cubic polynomial has $3$ real roots, among which $\tau^*$ is the only root in $\left(0,\frac{1}{2}\right).$\\
It is interesting that this value $\tau^*$ is independent of $a$ and $b$.\\
In the following result, $F_{\tau}$ is the quadratic polynomial interpolating $H_t(a,b)$ at $0,\tau,1$, as in \eqref{definition_F}.

\begin{theorem}
Let $a,b>0$ and let $H(t):=H_t(a,b)$ represent the Heinz means of $a,b$. If $F_{\tau}$ is the quadratic interpolator of $H_t$, then
$$\min_{\tau}\|H-F_{\tau}\|_1:=\min_{\tau}\int_{0}^{1}|H(t)-F_{\tau}(t)|\;dt$$ is attained at $\tau^*$, the unique root of $8\tau^3-12\tau^2+1=0$ in $\left(0,\frac{1}{2}\right).$ Moreover, since $H_t$ and $F_{\tau}$ are symmetric about $t=\frac{1}{2},$ this minimum is also attained at $1-t^*.$
\end{theorem}
\begin{proof}
Without loss of generality, we may assume $a=1.$ Since both $H$ and $F_{\tau}$ are symmetric about $t=\frac{1}{2}$, it suffices to investigate the integral over $\left[0,\frac{1}{2}\right].$ Moreover, it suffices to consider $\tau\in \left[0,\frac{1}{2}\right].$ Therefore, we are searching for $\tau^*\in \left[0,\frac{1}{2}\right]$ such that
$G(\tau):=\int_{0}^{1/2}|H(t)-F_{\tau^*}(t)|\;dt$ is minimum. By our remark, which followed Corollary \ref{full_comp_heinz}, we have
$H(t)\leq F_{\tau}(t)\;{\text {when}}\; t\leq \tau\;{\text{and}}\;H(t)\geq F_{\tau}(t)\;{\text{when}}\;\tau\leq t\leq\frac{1}{2}.$
Therefore, for $0<\tau<\frac{1}{2},$
\begin{eqnarray*}
G(\tau)&=&\int_{0}^{\frac{1}{2}}|H(t)-F_{\tau}(t)|dt\\
&=&\int_{0}^{\tau}\left(F_{\tau}(t)-H(t)\right)dt+\int_{\tau}^{\frac{1}{2}}\left(H(t)-F_{\tau}(t)\right)dt.
\end{eqnarray*}
Then
\begin{eqnarray*}
G'(\tau)&=&\left\{F_{\tau}(\tau)+\int_{0}^{\tau}\frac{\partial F_{\tau}(t)}{\partial \tau}dt-H(\tau)\right\}+\left\{F_{\tau}(\tau)-\int_{\tau}^{\frac{1}{2}}\frac{\partial F_{\tau}(t)}{\partial \tau}dt-H(\tau)\right\}\\
&=&\int_{0}^{\tau}\frac{\partial F_{\tau}(t)}{\partial \tau}dt-\int_{\tau}^{\frac{1}{2}}\frac{\partial F_{\tau}(t)}{\partial \tau}dt,
\end{eqnarray*}
where the last equation follows noting that $F_{\tau}(\tau)=H(\tau).$ Calculus computations imply
$$G'(\tau)=\frac{b^{-\tau}(8\tau^3-12\tau^2+1)}{24\tau^2(1-\tau)^2}f(\tau),$$
where
\begin{eqnarray*}
f(\tau)&=&(1-b^{\tau})(b^{\tau}-b)(2\tau-1)+(b^{2\tau}-b)(\tau-1)\tau \log b\\
&=&(b-b^{2\tau})g(\tau),
\end{eqnarray*}
for $$g(\tau)=\frac{(1-b^{\tau})(b^{\tau}-b)(2\tau-1)}{b-b^{2\tau}}-\tau(\tau-1)\log b.$$ We assert that $f\leq 0.$ Notice first that
$$g'(\tau)=\frac{(b^{\tau}-1)(b-b^{\tau})}{(b-b^{2\tau})^2}h(\tau),$$ where
$$h(\tau)=2b-2b^{2\tau}+(b+b^{2\tau})(2\tau-1)\log b.$$ Further,
$$h'(\tau)=2k(\tau)\log b\;{\text{where}}\;k(\tau)=b+b^{2\tau}(-1+(2\tau-1)\log b)$$ and $$k'(\tau)=2b^{2\tau}(2\tau-1)\log^2 b.$$

If $b>1$, then $k'(\tau)\leq 0$ because $\tau\leq\frac{1}{2},$ hence $k$ is decreasing and $k(\tau)\geq k\left(\frac{1}{2}\right)=0.$ Hence, $h'(\tau)\geq 0$ and $h(\tau)\leq h\left(\frac{1}{2}\right)=0$, which implies $g'(\tau)\leq 0$ and $g(\tau)\leq g(0)=0.$ Since $b>1,$ it follows that $f(\tau)\leq 0.$

If $b<1,$ then $k'(\tau)\leq 0$ and $k(\tau)\geq k\left(\frac{1}{2}\right)=0.$ Hence $h'(\tau)\leq 0$ and $h(\tau)\geq h\left(\frac{1}{2}\right)=0,$ which then implies $g'(\tau)\geq 0$ and $g(\tau)\geq g(0)=0.$ Since $0<b<1$, it follows that $f(\tau)\leq 0.$

Now let $\tau^*$ be the root of $\ell(\tau):=8\tau^3-12\tau^2+1$ in $\left(0,\frac{1}{2}\right),$ and notice that $\ell(\tau)\geq 0$ when $0\leq\tau\leq\tau^*$ and $\ell(\tau)\leq 0$ when $\tau^*\leq\tau\leq \frac{1}{2}.$\\
Since $f(\tau)\leq 0$ for all $0<\tau<1,$ and  $$G'(\tau)=\frac{b^{-\tau}(8\tau^3-12\tau^2+1)}{24\tau^2(1-\tau)^2}f(\tau),$$ it follows that $G'(\tau)\leq 0$ for $0<\tau\leq \tau^*$ and $G'(\tau)\geq 0$ for $\tau^*\leq \tau\leq \frac{1}{2}.$ This shows that $G$ attains its minimum at $\tau=\tau^*,$ which completes the proof.
\end{proof}

\subsection{Matrix versions}
In this part of the paper, we present some interesting matrix versions, based on the above scalar results. We emphasize that the significance of these results is the quadratic behavior of the refining terms, unlike the known results in the literature where linear refining terms have been discussed only, except in \cite{krnic}.

\subsubsection{Unitarily invariant norm versions}

The following is a quadratic refinement and reverse of the Heinz inequality
$$\|A^{\nu}XB^{1-\nu}+A^{1-\nu}XB^{\nu}\|_2\leq \|AX+XB\|_2, A,B\in\mathbb{M}_n^{+}, X\in\mathbb{M}_n.$$
\begin{theorem}\label{first_matrix}
Let $A,B\in\mathbb{M}_n^{+}, X\in\mathbb{M}_n$ and let $0< \nu,\tau< 1.$ If $\nu\leq r(\tau)$ or $\nu\geq R(\tau)$, then
\begin{align*}
\|A^{\nu}XB^{1-\nu}+A^{1-\nu}XB^{\nu}\|_2^2&+\frac{\nu(1-\nu)}{\tau(1-\tau)}\left[\|AX+XB\|_2^2-\|A^{\tau}XB^{1-\tau}+A^{1-\tau}XB^{\tau}\|_2^2\right]\\
&\leq \|AX+XB\|_2^2.
\end{align*}
If $r(\tau)\leq \nu\leq R(\tau)$, the inequality is reversed.
 \end{theorem}
\begin{proof}
 Let $A=U \Gamma U^*$ and $B=V \Lambda V^*$ be the spectral decomposition of $A$ and $B$. That is, $U,V$ are unitary matrices and $\Gamma, \Lambda$ are
the diagonal matrices whose diagonal entries are the eigenvalues $\{\lambda_i\}$ of $A$ and the eigenvalues $\{\mu_i\}$ of $B$, respectively. Denoting
$U^*XV$ by $Y$ and using $\circ$ for the Schur product, we have
$$A^{t}XB^{1-t}+A^{1-t}XB^{t}=U\left([\lambda_i^{t}\mu_j^{1-t}+\lambda_i^{1-t}\mu_j^{t}]\circ [y_{ij}]\right)V^*$$
and $AX+XB=U\left([\lambda_i+\mu_j]\circ [y_{ij}]\right)V^*.$ Now if $\nu\leq r(\tau)$ or $\nu\geq R(\tau)$ we have
\begin{align*}
 \|A^{\nu}&XB^{1-\nu}+A^{1-\nu}XB^{\nu}\|_2^2\\
&=\sum_{i,j}\left(\lambda_i^{\nu}\mu_j^{1-\nu}+\lambda_i^{1-\nu}\mu_j^{\nu}\right)^2|y_{ij}|^2\;({\text{Now\;apply\;}}\eqref{square_heinz_eq}\\
&\leq \sum_{i,j}(\lambda_i+\mu_j)^2|y_{ij}|^2-\frac{\nu(1-\nu)}{\tau(1-\tau)}\sum_{i,j}\left\{(\lambda_i+\mu_j)^2-
\left(\lambda_i^{\tau}\mu_j^{1-\tau}+\lambda_i^{1-\tau}\mu_j^{\tau}\right)^2\right\}|y_{ij}|^2\\
&=\|AX+XB\|_2^2-\frac{\nu(1-\nu)}{\tau(1-\tau)}\left[\|AX+XB\|_2^2-\|A^{\tau}XB^{1-\tau}+A^{1-\tau}XB^{\tau}\|_2^2\right].
\end{align*}
This completes the proof when $\nu\leq r(\tau)$ or $\nu\geq R(\tau).$ The other case follows similarly.
\end{proof}
In particular, if $\tau=\frac{1}{2}$, we obtain the quadratic refinement
$$\|A^{\nu}XB^{1-\nu}+A^{1-\nu}XB^{\nu}\|_2^2+4\nu(1-\nu)\left[\|AX+XB\|_2^2-\|A^{\frac{1}{2}}X B^{\frac{1}{2}}\|_2^2\right]\leq \|AX+XB\|_2^2,$$
which has been proved recently in \cite{krnic}.

Theorem \ref{first_matrix} has been proved by employing Proposition \ref{squared_heinz_prop}. A difference version maybe
 obtained by employing Corollary \ref{full_comp_heinz} as follows. The proof follows the same steps as Theorem \ref{first_matrix}, so we omit it.

\begin{theorem}\label{second_matrix}
Let $A,B\in\mathbb{M}_n^{+}, X\in\mathbb{M}_n$ and let $0< \nu,\tau< 1.$ If $\nu\leq r(\tau)$ or $\nu\geq R(\tau)$, then
\begin{align*}
&\frac{\|(AX+XB)-(A^{\tau}XB^{1-\tau}+A^{1-\tau}XB^{\tau})\|_2}{\tau(1-\tau)}\\
&\leq \frac{\|(AX+XB)-(A^{\nu}XB^{1-\nu}+A^{1-\nu}XB^{\nu})\|_2}{\nu(1-\nu)}.
\end{align*}
If $r(\tau)\leq \nu\leq R(\tau)$, the inequality is reversed.
 \end{theorem}

On the other hand, a reverse of the Heinz inequality may be found using Corollary \ref{full_multi_arith_geo_cor} as follows.
\begin{theorem}
Let $A,B\in\mathbb{M}_n^{++}, X\in\mathbb{M}_n$ and let $0< \nu,\tau< 1.$ Then there exist two positive numbers $m\leq M,$ depending on $A,B$, such that
\begin{align*}
\|AX+XB\|_2\leq \left(\frac{M+m}{2\sqrt{m M}}\right)^{\frac{\nu(1-\nu)}{\tau(1-\tau)}}\|A^{\nu}XB^{1-\nu}+A^{1-\nu}XB^{\nu}\|_2
\end{align*}
if $\nu\leq r(\tau)$ or $\nu\geq R(\tau)$.
\end{theorem}
\begin{proof}
We adopt the notations of Theorem \ref{first_matrix}. Since $A,B\in\mathbb{M}_n^{++}$, it follows that $\lambda_i,\mu_i>0.$ Let $m=\min\limits_{1\leq i\leq n}\{\lambda_i,\mu_i\}$ and $M=\max\limits_{1\leq i\leq n}\{\lambda_i,\mu_i\}.$ Now
\begin{align*}
\|AX+XB\|_2^2&=\sum_{i,j}(\lambda_i+\mu_j)^2|y_{ij}|^2\\
&\leq \sum_{i,j}\left(\frac{\lambda_i+\mu_j}{2\sqrt{\lambda_i\mu_j}}\right)^{2\frac{\nu(1-\nu)}{\tau(1-\tau)}}
\left(\lambda_i^{\nu}\mu_j^{1-\nu}+\lambda_i^{1-\nu}\mu_j^{\nu}\right)^2|y_{ij}|^2\;({\text{by}}\;\eqref{full_multi_arith_geo})\\
&\leq\left(\frac{M+m}{2\sqrt{m M}}\right)^{2\frac{\nu(1-\nu)}{\tau(1-\tau)}}\|A^{\nu}XB^{1-\nu}+A^{1-\nu}XB^{\nu}\|_2^2,
\end{align*}
where the last line is obtained noting $m\leq\lambda_i,\mu_j\leq M.$ This completes the proof.
\end{proof}

The function $\tau\mapsto \left(\frac{M+m}{2\sqrt{m M}}\right)^{2\frac{\nu(1-\nu)}{\tau(1-\tau)}}$ attains its minimum at $\tau=\frac{1}{2}.$ In this case, the above theorem is optimal and we have
$$\|AX+XB\|_2\leq \left(\frac{M+m}{2\sqrt{m M}}\right)^{4\nu(1-\nu)}\|A^{\nu}XB^{1-\nu}+A^{1-\nu}XB^{\nu}\|_2, 0\leq\nu\leq 1.$$ We should remark that the constant $\left(\frac{M+m}{2\sqrt{m M}}\right)^2$ is called the Kantorovich constant and has appeared in many recent studies treating matrix means.

In particular, if there exist $m,M>0$ such that $mI\leq A,B\leq MI,$ the above result is valid.\\
In the above results, we have presented matrix versions using the Hilbert-Schmidt norm. The following weaker version is valid for any unitarily invariant norm. For the proof, we need to recall the matrix H\"{o}lder inequality \cite{k}
\begin{equation}\label{holder}
\vert\vert\vert A^{t}XB^{1-t}\vert\vert\vert \leq \vert\vert\vert AX\vert\vert\vert ^{t}\vert\vert\vert XB\vert\vert\vert ^{1-t}, A,B\in\mathbb{M}_n^{+}, X\in\mathbb{M}_n,
\end{equation}
for any unitarily invariant norm $\vert\vert\vert \cdot\vert\vert\vert .$
\begin{theorem}\label{them_unit_needed}
Let $A,B\in\mathbb{M}_n^{+}, X\in\mathbb{M}_n$ and $0< \nu,\tau< 1.$ If $\nu\leq r(\tau)$ or $\nu\geq R(\tau)$ then
\begin{align*}
&\vert\vert\vert A^{\nu}XB^{1-\nu}+A^{1-\nu}XB^{\nu}\vert\vert\vert \\
&+\frac{\nu(1-\nu)}{\tau(1-\tau)}\left[\left(\vert\vert\vert AX\vert\vert\vert +\vert\vert\vert XB\vert\vert\vert \right)-
\left(\vert\vert\vert AX\vert\vert\vert ^{\tau}\vert\vert\vert XB\vert\vert\vert ^{1-\tau}+\vert\vert\vert AX\vert\vert\vert ^{1-\tau}\vert\vert\vert XB\vert\vert\vert ^{\tau}\right)\right]\\
&\leq \vert\vert\vert AX\vert\vert\vert +\vert\vert\vert XB\vert\vert\vert ,
\end{align*}
for any unitarily invariant norm $\vert\vert\vert \cdot\vert\vert\vert $ on $\mathbb{M}_n.$ In particular, if $\tau=\frac{1}{2}$ then
\begin{align*}
 \vert\vert\vert A^{\nu}XB^{1-\nu}+A^{1-\nu}XB^{\nu}\vert\vert\vert +4\nu(1-\nu)\left(\sqrt{\vert\vert\vert AX\vert\vert\vert }-\sqrt{\vert\vert\vert XB\vert\vert\vert }\right)^2\\
 \leq \vert\vert\vert AX\vert\vert\vert +\vert\vert\vert XB\vert\vert\vert ,
\end{align*}
for all $\nu\in [0,1]$.
\end{theorem}
\begin{proof}
When $\nu\leq r(\tau)$ or $\nu\geq R(\tau)$, we have
\begin{align*}
&\vert\vert\vert A^{\nu}XB^{1-\nu}+A^{1-\nu}XB^{\nu}\vert\vert\vert \\
&\leq \vert\vert\vert A^{\nu}XB^{1-\nu}\vert\vert\vert +\vert\vert\vert A^{1-\nu}XB^{\nu}\vert\vert\vert \;({\text{by\;the\;triangly\;inequality}})\\
&\leq \vert\vert\vert AX\vert\vert\vert ^{\nu}\vert\vert\vert XB\vert\vert\vert ^{1-\nu}+\vert\vert\vert AX\vert\vert\vert ^{1-\nu}\vert\vert\vert XB\vert\vert\vert ^{\nu}\;({\text{by}}\;\eqref{holder})\\
&\leq \left(\vert\vert\vert AX\vert\vert\vert +\vert\vert\vert XB\vert\vert\vert \right)\\
&\quad -\frac{\nu(1-\nu)}{\tau(1-\tau)}\left[\left(\vert\vert\vert AX\vert\vert\vert +\vert\vert\vert XB\vert\vert\vert \right)-
\left(\vert\vert\vert AX\vert\vert\vert ^{\tau}\vert\vert\vert XB\vert\vert\vert ^{1-\tau}+\vert\vert\vert AX\vert\vert\vert ^{1-\tau}\vert\vert\vert XB\vert\vert\vert ^{\tau}\right)\right],
\end{align*}
where we have used Corollary \ref{full_comp_heinz} to obtain the last inequality. This completes the proof.
\end{proof}
The case $\tau=\frac{1}{2}$ of the above Theorem has been shown in \cite{krnic}.

When $\nu=\frac{1}{2}$, the second inequality of Theorem \ref{them_unit_needed} is equivalent to the matrix Cauchy-Schwarz inequality $\||A^{\frac{1}{2}}XB^{\frac{1}{2}}\||\leq \||AX\||^{\frac{1}{2}}\||XB\||^{\frac{1}{2}}$, which is the case $t=\frac{1}{2}$ of \eqref{holder}. It should be mentioned here that this inequality can be concluded from the matrix arithmetic-geometric mean inequality
\begin{equation}\label{arith-geo_mat}
2\||A^{\frac{1}{2}}XB^{\frac{1}{2}}\||\leq \||AX+XB\||,
 \end{equation}
 which is the case $t=\frac{1}{2}$ of \eqref{matrix_heinz_intro}, as follows: In the inequality \eqref{arith-geo_mat}, replacing $A$ and $B$ by $tA$ and $\frac{1}{t}B$, respectively, where $t>0$, and using the triangle inequality, we have
 $$2\||A^{\frac{1}{2}}XB^{\frac{1}{2}}\||\leq t \||AX\||+\frac{1}{t}\||XB\||.$$
 Since this is true for all $t>0$, it follows that
 \begin{align*}
 2\||A^{\frac{1}{2}}XB^{\frac{1}{2}}\||&\leq \min_{t>0}\left(t \||AX\||+\frac{1}{t}\||XB\||\right)\\
 &=2\||AX\||^{\frac{1}{2}}\||XB\||^{\frac{1}{2}},
 \end{align*}
 which means $\||A^{\frac{1}{2}}XB^{\frac{1}{2}}\||\leq \||AX\||^{\frac{1}{2}}\||XB\||^{\frac{1}{2}}.$

 In view of the inequalities of Theorem \ref{them_unit_needed} and the inequality \eqref{arith-geo_mat}, it is reasonable to conjecture that
 $$2\||A^{\frac{1}{2}}XB^{\frac{1}{2}}+\left(\sqrt{\||AX\||}-\sqrt{\||XB\||}\right)^2\leq \||AX+XB\||,$$ which is a refinement of \eqref{arith-geo_mat}. However, this inequality is refuted by considering the two-dimensional example $A=\left[\begin{array}{cc}1&0\\ 0&\frac{3}{2}\end{array}\right], X=\left[\begin{array}{cc}1&0\\ 0&1\end{array}\right]$ and $B=\left[\begin{array}{cc}1&0\\ 0&\frac{1}{2}\end{array}\right].$ In this case, under the spectral norm $\|\;.\;\|$, we have
 \begin{align*}
 2\|A^{\frac{1}{2}}XB^{\frac{1}{2}}\|&+\left(\sqrt{\|AX\|}-\sqrt{\|XB\|}\right)^2=2+\left(\sqrt{\frac{3}{2}}-1\right)^2>2=\|AX+XB\|.
 \end{align*}

\subsubsection{Trace and determinant versions}
On the other hand, trace versions maybe obtained as follows. For the proof, we need to remind the reader of some facts about the trace.
 Recall that when $A,B\in\mathbb{M}_n^{+},$ one has
\begin{equation}\label{trace_ineq_needed_1}
\tr(A^{t}B^{1-t})\leq \tr^{t}(A)\tr^{1-t}(B), 0\leq t\leq 1.
\end{equation}
This inequality follows by log-convexity of the function $t\mapsto \tr(A^{t}B^{1-t}),$ \cite{Bou,saboam}. We present the following reverse that we need to prove our next result.
\begin{lemma}
Let Let $A,B\in\mathbb{M}_n^{++}$ and let $0\leq t\leq 1.$ Then
\begin{equation}\label{reverse_trace}
\tr(A^{t}B^{1-t})\left(\frac{\tr A\cdot\tr B}{\tr^2(A^{1/2}B^{1/2})}\right)^{R(t)}\geq \tr^{t}A\cdot\tr^{1-t}B,
\end{equation}
where $R(t)=\max\{t,1-t\}.$
\end{lemma}
\begin{proof}
Let $f(t)=\tr(A^{t}B^{1-t}).$ Then $f$ is log-convex. For $0\leq t\leq \frac{1}{2},$ notice that
$$\frac{1}{2}=\alpha t+(1-\alpha)\;{\text{where}}\;\alpha=\frac{1}{2-2t}.$$ Using log-convexity of $f$, we have
\begin{align*}
f\left(\frac{1}{2}\right)&\leq f^{\alpha}(t)f^{1-\alpha}(1).
\end{align*}
simplifying this inequality implies the result for $0\leq t\leq \frac{1}{2}.$ Similar computations yield the result for $\frac{1}{2}\leq t\leq 1.$
\end{proof}

\begin{theorem}
Let $A,B\in\mathbb{M}_n^{++}$ and let $0<\nu,\tau<1.$ If $\nu\leq r(\tau)$ or $\nu\geq R(\tau)$, then
\begin{align*}
&\tr\left(A^{\nu}B^{1-\nu}+A^{1-\nu}B^{\nu}\right)\\
&+\frac{\nu(1-\nu)}{\tau(1-\tau)}\left[\tr(A+B)-\left(\frac{\tr A\cdot \tr B}{\tr^2\left(A^{\frac{1}{2}}B^{\frac{1}{2}}\right)}\right)^{R(\tau)}\tr\left(A^{\tau}B^{1-\tau}+A^{1-\tau}B^{\tau}\right)\right]\\
&\leq \tr(A+B).
\end{align*}
On the other hand, if $r(\tau)\leq \nu\leq R(\tau)$, then
\begin{align*}
tr(A+B)&\leq \left(\frac{\tr A\cdot \tr B}{\tr^2\left(A^{\frac{1}{2}}B^{\frac{1}{2}}\right)}\right)^{R(\nu)}\tr\left(A^{\nu}B^{1-\nu}+A^{1-\nu}B^{\nu}\right)\\
&+\tr\left[A+B-\left(A^{\tau}B^{1-\tau}+A^{1-\tau}B^{\tau}\right)\right].
\end{align*}
\end{theorem}
\begin{proof}
If $\nu\leq r(\tau)$ or $\nu\geq R(\tau),$ then Corollary \ref{full_comp_heinz} implies
$$a+b\geq (a^{\nu}b^{1-\nu}+a^{1-\nu}b^{\nu})+\frac{\nu(1-\nu)}{\tau(1-\tau)}(a+b-(a^{\tau}b^{1-\tau}+a^{1-\tau}b^{\tau})),$$ for $a,b>0.$ In particular, let $a=\tr A$ and $b=\tr B,$ then apply \eqref{trace_ineq_needed_1} and \eqref{reverse_trace} to obtain
\begin{align*}
 \tr(A+B)&\geq \tr^{\nu}A\cdot\tr^{1-\nu}B+\tr^{1-\nu}A\cdot\tr^{\nu}B\\
&+\frac{\nu(1-\nu)}{\tau(1-\tau)}\left[\tr(A+B)-(\tr^{\tau}A\tr^{1-\tau}B+\tr^{1-\tau}A\tr^{\tau}B)\right]\\
&\geq \tr\left(A^{\nu}B^{1-\nu}+A^{1-\nu}B^{\nu}\right)\\
&+\frac{\nu(1-\nu)}{\tau(1-\tau)}\left[\tr(A+B)-\left(\frac{\tr A\cdot \tr B}{\tr^2\left(A^{\frac{1}{2}}B^{\frac{1}{2}}\right)}\right)^{R(\tau)}\tr\left(A^{\tau}B^{1-\tau}+A^{1-\tau}B^{\tau}\right)\right].
\end{align*}
The other inequality follows similarly.
\end{proof}

Our next result is a determinant version, where quadratic refinements are provided.
\begin{theorem}\label{det_thm}
Let $A,B\in\mathbb{M}_n^{++}$ and let $0<\nu,\tau<1.$ If $\nu\leq r(\tau)$ or $\nu\geq R(\tau),$ then
\begin{align*}
\det(A+B)^{\frac{1}{n}}&\geq \det(A\#_{\nu}B+A\#_{1-\nu}B)^{\frac{1}{n}}\\
&+\frac{\nu(1-\nu)}{\tau(1-\tau)}\det(A+B-(A\#_{\tau}B+A\#_{1-\tau}B))^{\frac{1}{n}}.
\end{align*}
\end{theorem}
\begin{proof}
Let $X=A^{-\frac{1}{2}}BA^{-\frac{1}{2}}$ and let $\lambda_i$ denote the $i$-th eigenvalue of $X$. Then noting Corollary \ref{full_comp_heinz} and the Minkowski inequality
$$\left(\prod_{i=1}^{n}a_i\right)^{\frac{1}{n}}+\left(\prod_{i=1}^{n}b_i\right)^{\frac{1}{n}}\leq \left(\prod_{i=1}^{n}(a_i+b_i)\right)^{\frac{1}{n}},a_i,b_i>0,$$ we have
\begin{align*}
 \det(I+X)^{\frac{1}{n}}&=\left(\prod_{i=1}^{n}\lambda_i(I+X)\right)^{\frac{1}{n}}\\
&=\left(\prod_{i=1}^{n}(1+\lambda_i(X))\right)^{\frac{1}{n}}\\
&\geq \prod_{i=1}^{n}\left[(\lambda_i^{\nu}+\lambda_i^{1-\nu})+\frac{\nu(1-\nu)}{\tau(1-\tau)}(1+\lambda_i-(\lambda_i^{\tau}+\lambda_i^{1-\tau}))\right]^{\frac{1}{n}}\\
&\geq \prod_{i=1}^{n}\left[(\lambda_i^{\nu}+\lambda_i^{1-\nu})\right]^{\frac{1}{n}}
+\prod_{i=1}^{n}\left[\frac{\nu(1-\nu)}{\tau(1-\tau)}(1+\lambda_i-(\lambda_i^{\tau}+\lambda_i^{1-\tau}))\right]^{\frac{1}{n}}\\
&=\left(\prod_{i=1}^{n}\lambda_i(X^{\nu}+X^{1-\nu})\right)^{\frac{1}{n}}
+\frac{\nu(1-\nu)}{\tau(1-\tau)}\left(\prod_{i=1}^{n}\lambda_i(I+X-(X^{\tau}+X^{1-\tau}))\right]^{\frac{1}{n}}\\
&= \det(X^{\nu}+X^{1-\nu})^{\frac{1}{n}}+\frac{\nu(1-\nu)}{\tau(1-\tau)}\det(I+X-(X^{\tau}+X^{1-\tau}))^{\frac{1}{n}}.
\end{align*}
Then multiplying both sides by $\det A$ and utilizing simple properties of the determinant, we get the required inequality.
\end{proof}
Notice that the above theorem provides a refinement of the well known determinant inequality
$$\det(A\#_{\nu}B+A\#_{1-\nu}B)\leq \det(A+B).$$ In particular, when $\tau=\frac{1}{2},$ Theorem \ref{det_thm} reads as follows
$$\det(A\#_{\nu}B+A\#_{1-\nu}B)^{\frac{1}{n}}+4\nu(1-\nu)\det(A+B-2\;A\#B)^{\frac{1}{n}}\leq \det(A+B)^{\frac{1}{n}}, 0\leq \nu\leq 1.$$

Following the proof of Theorem \ref{det_thm} and using Proposition \ref{squared_heinz_prop}, we obtain the following squared version for the determinant of the Heinz means.

\begin{theorem}\label{det_thm_square}
Let $A,B\in\mathbb{M}_n^{++}$ and let $0<\nu,\tau<1.$ If $\nu\leq r(\tau)$ or $\nu\geq R(\tau),$ then
\begin{align*}
\det(A+B)^{\frac{2}{n}}&\geq \det(A\#_{\nu}B+A\#_{1-\nu}B)^{\frac{2}{n}}\\
&+\left(\frac{\nu(1-\nu)}{\tau(1-\tau)}\right)^2\det(A+B-(A\#_{\tau}B+A\#_{1-\tau}B))^{\frac{2}{n}}.
\end{align*}
\end{theorem}

The above are additive determinant versions. An interesting multiplicative version can be found using Corollary \ref{full_multi_arith_geo_cor}. The proof is similar to that of Theorem \ref{det_thm}, and hence is left to the reader.
\begin{theorem}
Let $A,B\in\mathbb{M}_n^{++}$ and let $0<\nu,\tau<1.$ If $\nu\leq r(\tau)$ or $\nu\geq R(\tau),$ then
\begin{align*}
\left(\frac{\det(A+B)}{\det(A\#_{\nu}B+A\#_{1-\nu}B)}\right)^{\frac{1}{\nu(1-\nu)}}\leq \left(\frac{\det(A+B)}{\det(A\#_{\tau}B+A\#_{1-\tau}B)}\right)^{\frac{1}{\tau(1-\tau)}}
\end{align*}
On the other hand, if $r(\tau)\leq\nu\leq R(\tau)$, the inequality is reversed.
\end{theorem}
In particular, when $\tau=\frac{1}{2},$ we have the following multiplicative reverse of the determinant of the Heinz means
\begin{align}\label{det_reverse_multi_needed}
\det(A+B)\leq \left(\frac{\det(A+B)}{2^n\sqrt{\det(AB)}}\right)^{4\nu(1-\nu)}\det(A\#_{\nu}B+A\#_{1-\nu}B), 0\leq \nu\leq 1.
\end{align}
Notice that $4\nu(1-\nu)\leq 1, \leq \nu\leq 1.$ In this case, a weaker version of \eqref{det_reverse_multi_needed} is as follows
\begin{align*}
\sqrt{\det(AB)}\leq \det\frac{A\#_{\nu}B+A\#_{1-\nu}B}{2},
\end{align*}
which is the determinant version of the first inequality in \eqref{original_heinz_intro}.

\subsubsection{L\"{o}wener partial ordering}

Our final goal in this article is to present some matrix versions using the strongest comparison; the L\"{o}wener partial ordering. Recall that for two Hermitian matrices $A,B\in\mathbb{M}_n$, the notation $A\leq B$ is used to mean $B-A\in\mathbb{M}_n^{+}.$ This introduces a partial ordering on positive matrices and is considered as the strongest comparison. More precisely, when $A,B\in\mathbb{M}_n^{+}$ are such that $A\leq B$, one concludes that $\lambda_i(A)\leq \lambda_i(B),$ where $\lambda_i(X)$ is the $i-$th eigenvalue of $X$, when written in a decreasing order. Then the relation $\lambda_i(A)\leq \lambda_i(B)$ implies that $\tr A\leq \tr B, \det A\leq \det B$ and $\vert\vert\vert A\vert\vert\vert \leq\vert\vert\vert B\vert\vert\vert ,$ for any unitarily invariant norm on $\mathbb{M}_n.$ In this section, we use the notation $$H_t(A,B)=\frac{A\#_{t}B+A\#_{1-t}B}{2}, A,B\in\mathbb{M}_n^{++}.$$ A standard functional calculus argument applied on \eqref{original_heinz_intro} implies the following matrix version
\begin{equation}\label{matr_heinz_order}
A\#B\leq H_t(A,B)\leq A\nabla B, A,B\in\mathbb{M}_n^{++}, 0\leq t\leq 1.
\end{equation}
In the following theorem, we present a quadratic refinement and reverse of this inequality.
\begin{theorem}
Let $A,B\in\mathbb{M}_n^{++}$ and let $0<\nu,\tau<1.$ If $\nu\leq r(\tau)$ or $\nu\geq R(\tau),$ then
$$\frac{A\nabla B-H_{\tau}(A,B)}{\tau(1-\tau)}\leq \frac{A\nabla B-H_{\nu}(A,B)}{\nu(1-\nu)}.$$  On the other hand if $r(\tau)\leq\nu\leq R(\tau)$, the inequality is reversed.
\end{theorem}
The proof of this theorem follows a standard argument as in the next result.
\begin{theorem}
Let $A,B\in\mathbb{M}_n^{++}$ and let $0<\nu,\tau<1.$ If $\nu\leq r(\tau)$ or $\nu\geq R(\tau),$ then
$$\frac{A+BA^{-1}B-(A\#_{2\tau}B+A\#_{2-2\tau}B)}{\tau(1-\tau)}\leq \frac{A+BA^{-1}B-(A\#_{2\nu}B+A\#_{2-2\nu}B)}{\nu(1-\nu)}.$$  On the other hand if $r(\tau)\leq\nu\leq R(\tau)$, the inequality is reversed.
\end{theorem}
\begin{proof}
Letting $a=1$ in \eqref{second_heinz}, we get
$$\frac{1+b^2-(b^{2\tau}+b^{2-2\tau})}{\tau(1-\tau)}\leq \frac{1+b^2-(b^{2\nu}+b^{2-2\nu})}{\nu(1-\nu)}, b>0.$$
Now let $X=A^{-\frac{1}{2}}BA^{-\frac{1}{2}}.$ Then $X\in\mathbb{M}_n^{++}$. Therefore by applying monotonicity of continuous functions on Hermitian matrices, we get
$$\frac{I+X^2-(X^{2\tau}+X^{2-2\tau})}{\tau(1-\tau)}\leq \frac{I+X^2-(X^{2\nu}+X^{2-2\nu})}{\nu(1-\nu)}.$$ Conjugating both sides with $A^{\frac{1}{2}}$ implies the desired inequality when $\nu\leq r(\tau)$ or $\nu\geq R(\tau).$ The reversed version follows similarly.
\end{proof}
Notice that the above result allows comparison of means with parameters bigger than 1. This happens when $\nu,\tau>\frac{1}{2}$.

\end{document}